\documentclass[12pt]{article}
\usepackage[utf8]{inputenc}
\usepackage{amsmath}
\usepackage{amssymb}
\usepackage{amsthm}
\usepackage{graphicx}
\graphicspath{ {images/} }
\usepackage{esvect}

\newcommand{\abs}[1]{\lvert#1\rvert}

\newcommand{\C}{\mathbb{C}}
\newcommand{\bbC}{\mathbb{C}}

\newcommand{\ot}{\otimes}
\newcommand{\la}{\lambda}

\newcommand{\GL}{\text{GL}}

\usepackage{geometry}
\usepackage{tikz}
\theoremstyle{definition}

\newtheorem{proposition}{Theorem}
\newtheorem{lemma}{Lemma}
\theoremstyle{remark}
\newtheorem*{example}{Example}
\usepackage{dsfont}
\usepackage{tikz}
\usepackage{mathtools}
\usepackage[vcentermath]{youngtab}

\begin{document}

\title{Branching from the General Linear Group to the Symmetric Group and the Principal Embedding}

\author{Alexander Heaton, Songpon Sriwongsa, Jeb F. Willenbring}

\maketitle

\begin{abstract}
    Let $S$ be a principally embedded $\mathfrak{sl}_2$-subalgebra in $\mathfrak{sl}_n$ for $n \geq 3$. A special case of results of the third author and Gregg Zuckerman implies that there exists a positive integer $b(n)$ such that for any finite-dimensional irreducible $\mathfrak{sl}_n$-representation, $V$, there exists an irreducible $S$-representation embedding in $V$ with dimension at most $b(n)$. In a 2017 paper (joint with Hassan Lhou), they prove that $b(n)=n$ is the sharpest possible bound, and also address embeddings other than the principal one.\\
    
    These results concerning embeddings may by interpreted as statements about \emph{plethysm}. Then, in turn, a well known result about these plethysms can be interpreted as a ``branching rule''.  Specifically, a finite dimensional irreducible representation of $\GL(n,\bbC)$ will decompose into irreducible representations of the symmetric group when it is restricted to the subgroup consisting of permutation matrices. The question of which irreducible representations of the symmetric group occur with positive multiplicity is the topic of this paper, applying the previous work of Lhou, Zuckerman, and the third author.
\end{abstract}

A complex irreducible representation $V$ of $\mathfrak{sl}_2(\C)$ defines a homomorphism
\begin{equation*}
    \pi: \mathfrak{sl}_2 \rightarrow \mbox{End}(V).
\end{equation*}
Fixing an ordered basis we obtain an identification $\text{End}(V) \cong \mathfrak{gl}_n$. Since $\mathfrak{sl}_2$ is a simple Lie algebra, the kernel is trivial and the image of $\pi$, denoted $\mathfrak{s}$, is therefore isomorphic to $\mathfrak{sl}_2$. We will refer to $\mathfrak{s}$ as a \emph{principal} $\mathfrak{sl}_2$-subalgebra of $\mathfrak{gl}_n$. In fact, since $\mathfrak{s}$ is simple it intersects the center of $\mathfrak{gl}_n$ trivially and hence $\mathfrak{s} \subseteq \mathfrak{sl}_n$ (except when $n=1$). There are other embeddings of $\mathfrak{sl}_2$ when $V$ is not irreducible, but we will only discuss the principal embedding in this paper.

Restricting the adjoint representation of a simple Lie algebra to a \emph{principal} $\mathfrak{sl}_2$-embedding, we can decompose and find multiplicities. In 1958, Bertram Kostant interpreted these multiplicities topologically in \cite{Kostant}, yielding the Betti numbers of the compact form of the corresponding Lie group. People have been interested in the principal embedding ever since. In future work we hope to consider the analogs of our results for other Lie types and other embeddings. In this paper we show a relationship between the principal embedding and branching from $GL_n$ to the symmetric group. Our main tool is the following theorem proved in \cite{LhouWillenbring} which was anticipated in \cite{WillenbringZuckerman}:
\begin{proposition}
    Fix $n \geq 3$ and a principal $\mathfrak{sl}_2$-subalgebra, $\mathfrak{s}$, of $\mathfrak{sl}_n$. Let $V$ denote an arbitrary finite dimensional complex irreducible representation of $\mathfrak{sl}_n$.  Then, there exists $0 \leq d < n$ such that upon restriction to $\mathfrak{s}$, V contains the irreducible $\mathfrak{s}$ representation $F_d$ in the decomposition.
\end{proposition}

The expository aspects of this paper should be put in context by mentioning some previous work, both old and new. Certainly, any work related to branching rules has benefited from the older extensive work of R. C. King, specifically \cite{King1974}. From another point of view, a well-known approach is to study the action of a Weyl group on the weight spaces in a finite dimensional irreducible representation of the corresponding Lie algebra. The special case of the zero weight space is of particular interest and addressed in \cite{Nishiyama} for the case of the symmetric group. More recently, the combinatorics of this problem are related to the stable Kronecker coefficients in \cite{OrellanaZabrocki2019}.

\textbf{Structure of this paper}: In this paper, we try to make some progress toward understanding the \textit{branching problem}: Can we describe how representations of $GL_n$ decompose upon restriction to the permutation matrices $\mathfrak{S}_n$? We attack the branching problem by realizing its equivalence to certain instances of \emph{plethysm}. Section \ref{sec:algorithm} describes a well-known algorithm that allows us to compute these branching multiplicities in any specific case. Section \ref{sec:motivation} provides some motivation by connecting branching with dynamical systems. Section \ref{sec:ConnectBranchingPlethysm} explains the connection between plethysm and branching. Section \ref{sec:OneRow} gives a known combinatorial description of branching for one-row shapes (symmetric powers). Finally, Section \ref{sec:finishProof} proves our main Theorem \ref{thm:Main}, which guarantees the existence of all $\mathfrak{S}_n$ irreducible representations inside certain two-row irreducible representations of $GL_n$. In the next section we provide a few definitions and notation followed by brief and explicit examples of the results of this paper, including a statement of the main theorem.

\section{Notation, brief examples, and main theorem}

We define a partition $\lambda$ of a nonnegative integer $a$ as a sequence $(\la_1,\dots,\la_k) \in \mathbb{N}^k$, satisfying $\la_1 \geq \cdots \geq \la_k > 0$ and $\sum \la_i = a$. We say such a $\la$ has $k$ \textit{parts} and \textit{size} $a$, writing $\ell(\la)=k$ and $|\la|=a$. Any $\la_i=0$ is considered irrelevant, so we could identify $\la$ with the infinite sequence $(\la_1,\dots,\la_k,0,0,\dots)$. For a partition $\la$ with at most $n$ parts let
$F^{\la}_n$ be the irreducible $GL_n$ representation with highest
weight indexed by $\la$. Throughout the paper we will sometimes write \textit{irrep} instead of \textit{irreducible representation}. If $\la$ has size $s$, let $Y^{\la}_s$ denote the irreducible complex representation of the symmetric group, $\mathfrak{S}_s$,
paired with $F^\la_n$ by Schur-Weyl duality (see for example \cite{Etingof} or \cite{Weyl1}) so that
\begin{equation*}
    {\bigotimes}^s \bbC^n \cong \bigoplus F^\la_n \ot Y^\la_s
\end{equation*}
where the sum is over all partitions of $s$ with at most $n$ parts (the symmetric group action commutes with the $GL_n$ diagonal action and the decomposition is multiplicity-free). We take this as our definition of $Y^\la_s$. The following example illustrates our main result.

\begin{example}
    Consider the symmetric group on 10 letters $\mathfrak{S}_{10}$. It's $42$ irreps $Y^\mu_{10}$ are in correspondence with partitions $\mu$ of size 10. Our main result shows that every irrep $Y_{10}^{\mu}$ has non-zero multiplicity in the decomposition of certain two-row partitions $\lambda$ of $10m$, for $m \in \{2,3,\dots\}$. %In this paper we will have results for each choice of $m \in \{2,3,\dots\}$.
    Choosing $m=2$ our results state that every irrep $Y^\mu_{10}$ appears in at least one of the following two irreps of $GL_{10}$:
    \begin{equation*}
        F_{10}^{(10,10)} \mbox{ or } F_{10}^{(11,9)}
    \end{equation*}
    Choosing $m=3$, we can guarantee that every irrep $Y^\mu_{10}$ of the symmetric group occurs in one of the following irreps of $GL_{10}$:
    \begin{equation*}
        F_{10}^{(15, 15)} \mbox{ or } F_{10}^{(16, 14)}
    \end{equation*}
    Choosing $m=4$ we find that every $Y^\mu_{10}$ appears with non-zero branching multiplicity inside at least one of
    \begin{equation*}
        F_{10}^{(20, 20)} \mbox{ or } F_{10}^{(21, 19)} \mbox{ or } F_{10}^{(22, 18)}.
    \end{equation*}
\end{example}

Given any irrep of $\mathfrak{S}_n$, we guarantee its non-zero multiplicity in certain \emph{short-tail} two-row irreps of the general linear group $GL_n$. The main Theorem \ref{thm:Main} reads as follows:

\textbf{Main Theorem}: Choose any irreducible representation $Y^\mu_n$ of $\mathfrak{S}_n$ and choose any $m \in \{2,3,4,\dots\}$. Consider the set of irreducible representations of $GL_n$ denoted $F^\la_n$ where $\la = (p+d,p)$, where $p,d$ are integer solutions of $2p+d=nm$ restricted to $d \in \{0,1,2,\dots,m\}, p \in \mathbb{N}$. Then the multiplicity
    \begin{equation*}
        \left[ Y^\mu_n : F^\la_n \right] \ne 0
    \end{equation*}
for at least one of the $F^\la_n$. 

\textbf{Remark:} Taking $m=2$ these irreps come near the boundary of a certain interesting phenomenon which we do not yet understand, and which our theorem does not explain. Consider Figure \ref{fig:curiousCurve}.
\begin{figure}[ht]
    \centering
    \includegraphics[width=\textwidth]{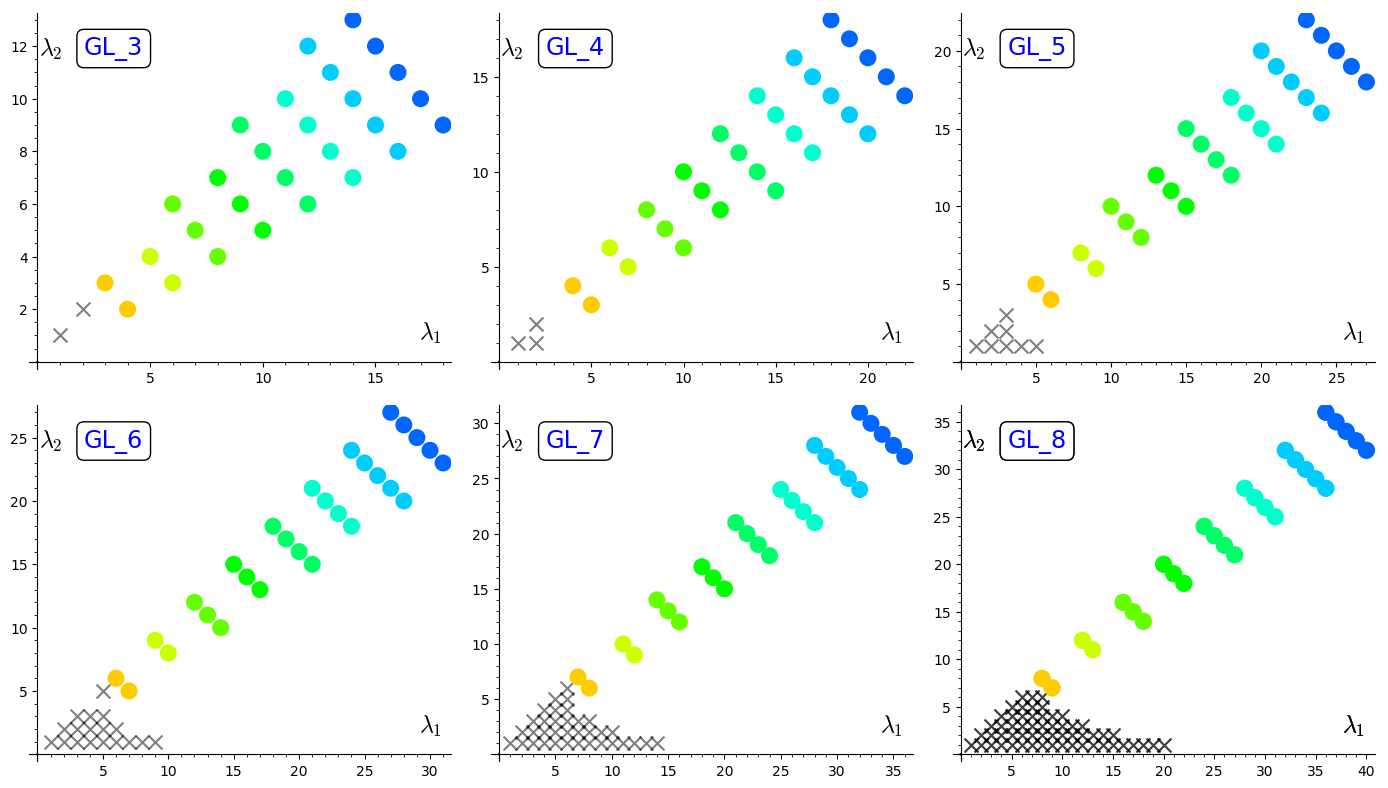}
    \caption{A curious curve emerges}
    \label{fig:curiousCurve}
\end{figure}
The colorful points represent $\la$ from our main Theorem \ref{thm:Main}, with each distinct color corresponding to a distinct choice of $m$. The phenomenon is a discrete curve (in the $\la_1, \la_2$ plane) dividing irreps of $GL_n$ which contain every symmetric group irrep from those that are missing at least one symmetric group irrep. The points marked with an X are the $GL_n$ irreps which, upon restriction to the symmetric group, fail to contain every $\mathfrak{S}_n$ irrep. As can also be seen in the figure, this discrete curve has a \textit{jump} for $GL_6$. In fact a similar jump occurs for $GL_{10}$ (not pictured). We do not know the pattern. Our results only capture the phenomenon in a very limited sense: Given a fixed irreducible representation, and a fixed color (choice of $m$) its multiplicity will be nonzero in at least one position marked by that color. Explaining the other points in these diagrams is an important topic for future research and outside the results of the current article.

\textbf{Another example of our results}: Below is a (partial) list of multiplicities of the irreps of $\mathfrak{S}_{10}$ which appear in the decomposition of the $GL_{10}$ irrep $(11,9)$ (only 12 of the 42 required numbers are listed). You'll notice the last multiplicity is zero. The irreducible $\mathfrak{S}_{10}$ representation indexed by $(1,1,1,1,1,1,1,1,1,1)$ does not occur in the decomposition.
\begin{equation*}
    4789, 25466,61323,88744,157620,\dots,676,2302,4058,2132,459,32,0
\end{equation*}
Our theorem predicts that we can find every irreducible representation of the symmetric group $\mathfrak{S}_{10}$ inside either $(11,9)$ or $(10,10)$. In fact, decomposing $(10,10)$ via the algorithm described in Section \ref{sec:algorithm} we indeed find the irrep $(1,1,1,1,1,1,1,1,1,1)$ occurring with multiplicity $1$. This is one example where our theorem finds the \textit{boundary} of the phenomenon depicted in Figure \ref{fig:curiousCurve}. Below the curve we have $GL_n$ irreps which are missing at least one $Y_n^\mu$. Above the curve we have $GL_n$ irreps using every $Y_n^\mu$. Our theorem, taking $m=2$, guarantees the appearance of every $Y_n^\mu$ in one of two $GL_n$ irreps near that boundary.

\textbf{Acknowledgement:} We would like to thank the organizers Mohammad Reza Darafsheh and Manouchehr Misaghian of the AMS Special Session on Group Representation Theory and Character Theory held January 19, 2019 at the Joint Math Meetings, where we presented this research. The authors would like to sincerely thank the two anonymous referees whose comments both improved the results and the exposition.

\section{An algorithm for branching}\label{sec:algorithm}

Branching from $GL_n$ to $\mathfrak{S}_n$ is among the class of problems which have an algorithm we can use to find the answer in any specific (finite) case, but unfortunately lacks a general description, formula, or combinatorial explanation. Already well-known is a combinatorial description in the special case of one-row diagrams (symmetric powers). For a description of this see Section \ref{sec:OneRow}. The results of this paper are therefore a step towards the next case: two-row diagrams $(\lambda_1,\lambda_2)$. We now give a brief description of the algorithm which, given any specific irrep, will output its decomposition.

\textbf{Algorithm}: We can decompose representations of the general linear group into irreps of the symmetric group by the following (roughly sketched) algorithm. The input is a symmetric function corresponding to the character of a fixed $GL_n$ representation. The output is a list of multiplicities for each irrep of $\mathfrak{S}_n$. The permutation matrices are a subgroup of $GL_n$ and, diagonalized, they have certain eigenvalues (roots of unity) corresponding to their cycle type. Diagonalizable elements are dense in $GL_n$ and so we know the character of an irreducible representation of the general linear group is given by evaluating a Schur function in $n$ variables corresponding to parameters of the maximal torus inside $GL_n$. Replacing these variables with the corresponding eigenvalues (of correct multiplicities) for a permutation matrix of each cycle type, we create the trace of the operator of an element of the symmetric group acting on that same vector space (the representation of $GL_n$ whose character we've taken). By doing this over all possible cycle types, we find the character viewed as a representation of the symmetric group. By taking the inner product with irreducible characters of the symmetric group we can find the multiplicities of each irreducible representation of $\mathfrak{S}_n$ inside the original $GL_n$ representation.

\section{Some motivation}\label{sec:motivation}
There are many reasons to study the decomposition of $GL_n$ representations under restriction to the symmetric group $\mathfrak{S}_n$. In this section we briefly present one reason, although we believe there are reasons yet to be discovered as well.

Repeatedly pressing the cosine button on your calculator is a good example of a dynamical system. Since your calculator presumably has finite memory, this is a dynamical system on a finite set. For example, Figure \ref{fig:ThreeElementSystems} provides a list of all 7 dynamical systems on a 3 element set.
\begin{figure}[h]
    \centering
    \includegraphics[scale=0.7]{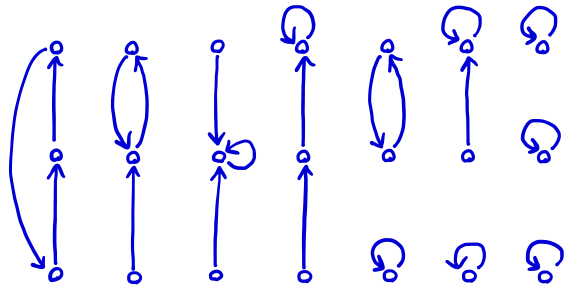}
    \caption{Dynamical systems on a three element set}
    \label{fig:ThreeElementSystems}
\end{figure}
Counting the number of such dynamical systems can (surprisingly!) be accomplished simply by summing up the appropriate branching multiplicities for the decomposition of a certain representation. We will briefly sketch this story for this particular example (counting the 7 dynamical systems listed in Figure \ref{fig:ThreeElementSystems}) although it applies to dynamical systems on any finite set.

Consider $\C^3 \otimes \C^3 \otimes \C^3$ under the action of the symmetric group $\mathfrak{S}_3$ permuting tensor factors, and one copy of $GL_3$ acting diagonally on each $\C^3$. This decomposes under Schur-Weyl duality as
\begin{equation*}
    \C^3 \otimes \C^3 \otimes \C^3 = \bigoplus F^\la \otimes Y_\la
\end{equation*}
Restricting to the permutation matrices $\mathfrak{S}_3$ sitting inside $GL_3$ the representation decomposes further with branching multiplicity coefficients we will call $b^\la_\mu$, which are of course non-negative integers and the subject of this paper. We write this as follows:
\begin{align*}
    \C^3 \otimes \C^3 \otimes \C^3 &= \bigoplus_\la F^\la \otimes Y_\la\\
     &= \bigoplus_\la \left( \bigoplus_\mu b^\la_\mu Y^\mu \right) \otimes Y_\la\\
     &= \bigoplus_{\la, \mu} b^\la_\mu \left( Y^\mu \otimes Y_\la \right)
\end{align*}
In order to find the relationship between these $b^\la_\mu$ and dynamical systems on a finite state space, first consider all functions from $\{1,2,3\}$ to $\{1,2,3\}$. Call them $X$, then there are $\abs{X} = 27=dim(\C^3 \otimes \C^3 \otimes \C^3)$ of them. In fact, consider one function $f \in X$ where $f$ sends $1 \mapsto 1, 2 \mapsto 1, 3 \mapsto 2$.
\begin{figure}[h]
    \centering
    \includegraphics[scale=0.4]{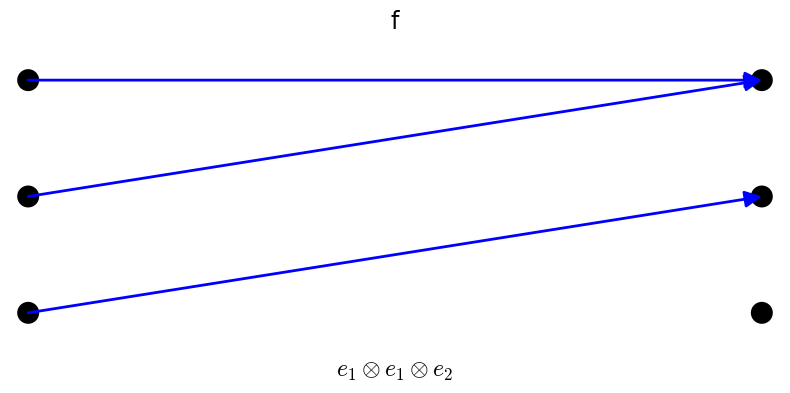}
    \label{fig:fPlot}
\end{figure}
This function corresponds to a basis element of $\C^3 \otimes \C^3 \otimes \C^3$, namely $e_1 \otimes e_1 \otimes e_2$. If $e_1,e_2,e_3$ are a basis we have chosen for $\C^3$ then a choice of one $e_i$ in each tensor factor is a choice of image $f(x) \in \{1,2,3\}$ for each element of the domain $x \in \{1,2,3\}$. The permutation matrices inside $GL_3$ are therefore acting by permuting the basis elements $e_1,e_2,e_3$ in the same way on each of the tensor factors. Thus they are permuting the choices of image. For example the permutation $\sigma=(2,3)$ would send $e_1\otimes e_1 \otimes e_2$ to $e_1 \otimes e_1 \otimes e_3$ which corresponds to an action on the function $f$, sending it to $\sigma.f$: $1 \mapsto 1, 2 \mapsto 1, 3 \mapsto 3$. The other copy of $\mathfrak{S}_3$ which is acting by permuting tensor factors (rather than sitting inside the $GL_3$) acts differently, simply by permuting the domain. For example, $\sigma$ acts by sending $e_1\otimes e_1 \otimes e_2$ to $e_1 \otimes e_2 \otimes e_1$. This corresponds to an action on the function $f \mapsto \sigma..f$ where $\sigma..f(x)=f(\sigma^{-1}x)$ so that $\sigma..f$ sends $1 \mapsto 1, 2 \mapsto 2, 3 \mapsto 1$.

If we consider $\Delta \mathfrak{S}_3 \subset \mathfrak{S}_3 \times \mathfrak{S}_3$ (the diagonal subgroup: take the same group element in both factors of the direct product) acting on $X$ then $X$ splits into orbits
\begin{equation*}
    X = \mathcal{O}_1 \cup \cdots \cup \mathcal{O}_r
\end{equation*}
where each orbit corresponds to one dynamical system. Thus, if we can count the orbits, we have counted the dynamical systems on 3 points. To see this, realize that $\Delta \mathfrak{S}_3$ corresponds to letting both copies of the symmetric group act in the same way on the domain and codomain. The function from $\{1,2,3\}$ to $\{1,2,3\}$ collapses and becomes a dynamical system on 3 points. In fact, many different functions collapse to the same dynamical system, namely all functions in the same orbit of $\Delta \mathfrak{S}_3$. For example, consider Figure \ref{fig:Collapse} for a depiction of this collapse for our function $f$:
\begin{figure}[h]
    \centering
    \includegraphics[scale=0.6]{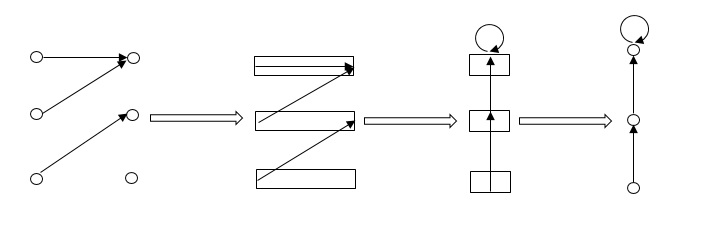}
    \caption{Our function becomes a dynamical system}
    \label{fig:Collapse}
\end{figure}
Now, how do we count orbits? Instead of using Burnside's formula which averages the number of fixed points over the group, we can also use functions on $X$ and simply count the $\Delta \mathfrak{S}_3$-fixed vectors since
\begin{align*}
    \mbox{dim } \C[X]^{\Delta \mathfrak{S}_3} &= \mbox{dim } \C[\mathcal{O}_1 \cup \cdots \cup \mathcal{O}_r ]^{\Delta \mathfrak{S}_3}\\
     &= \mbox{dim } \C[\mathcal{O}_1]^{\Delta \mathfrak{S}_3} \oplus \cdots \oplus \C[\mathcal{O}_r]^{\Delta \mathfrak{S}_3}
\end{align*}
There will be one linearly-independent $\Delta \mathfrak{S}_3$-fixed vector per orbit, namely the sum of basis elements taken to be delta functions on each point in the orbit under consideration. But we also know that
\begin{align*}
    \C[X]^{\Delta \mathfrak{S}_3} &= \left( \C^3 \otimes \C^3 \otimes \C^3 \right)^{\Delta \mathfrak{S}_3}\\
     &= \left( \bigoplus_{\la, \mu} b^\la_\mu \left( Y^\mu \otimes Y_\la \right) \right)^{\Delta \mathfrak{S}_3}\\
     &= \sum b^\la_\la.
\end{align*}
The last step above is explained by observing that the trivial representation of $\Delta \mathfrak{S}_3$ occurs exactly once in every copy of $Y^\mu \otimes Y_\la$ where $\mu=\la$ and zero times elsewhere. Since irreps of the symmetric group are self-dual, we have:
\begin{align*}
    \left( Y^\mu \otimes Y_\la \right)^{\Delta \mathfrak{S}_3} &= \left( (Y^\mu)^\ast \otimes Y_\la \right)^{\Delta \mathfrak{S}_3}\\
     &= \text{Hom} \left( Y^\mu, Y_\la \right)^{\Delta \mathfrak{S}_3}\\
     &= \text{Hom}_{\Delta \mathfrak{S}_3} \left( Y^\mu, Y_\la \right)
\end{align*}
which is clearly 0 or 1, depending on if $\mu = \la$. This shows that we can compute the number of dynamical systems simply by summing branching multiplicities $\sum b^\la_\la$.

Now we finish finding the answer 7 by adding up all possible $b^\la_\la$ for $\la \in \{ \tiny\yng(3), \tiny\yng(2,1), \tiny\yng(1,1,1) \}$. Consider first $b^{\tiny\yng(3)}_{\tiny\yng(3)}$, which corresponds to decomposing the $GL_3$ representation $F^{\tiny\yng(3)}$ of degree 3 homogeneous polynomials in 3 variables. Finding the multiplicity of the trivial representation $Y^{\tiny\yng(3)}$ of the symmetric group is the same as finding the number of linearly independent $\mathfrak{S}_3$-fixed vectors. These are clearly
\begin{equation*}
    x^3+y^3+z^3,\hspace{1cm} x^2y+x^2z+y^2x+y^2z+z^2x+z^2y, \hspace{1cm} xyz
\end{equation*}
and so the coefficient $b^{\tiny\yng(3)}_{\tiny\yng(3)}=3$.

Consider $\tiny\yng(2,1)$. As a representation of $GL_3$ this is sometimes referred to as the \textit{eightfold way} or \textit{octet} representation, since it is also a representation of the subgroup $SU(3)$ and finds application in particle physics. This representation decomposes as:
\begin{equation*}
    F^{\tiny\yng(2,1)} = \tiny\yng(3) + 3\tiny\yng(2,1) + \tiny\yng(1,1,1)
\end{equation*}
which means that $b^{\tiny\yng(2,1)}_{\tiny\yng(2,1)} = 3$. Lastly, consider $\tiny\yng(1,1,1)$. As a $GL_3$ representation this is the determinant, and upon restriction to $\mathfrak{S}_3$ we do in fact obtain the sign representation $Y^{\tiny\yng(1,1,1)}$ with multiplicity 1, so $b^{\tiny\yng(1,1,1)}_{\tiny\yng(1,1,1)} = 1$. Thus we have our result:
\begin{align*}
    b^{\tiny\yng(3)}_{\tiny\yng(3)} + b^{\tiny\yng(2,1)}_{\tiny\yng(2,1)} + b^{\tiny\yng(1,1,1)}_{\tiny\yng(1,1,1)} &= \text{Number of dynamical systems on 3 points.}\\
    3 + 3 + 1 &= 7.
\end{align*}

\section{Connecting branching with plethysm}\label{sec:ConnectBranchingPlethysm}
Here we start to prove the results of this paper. First we will show that a certain branching multiplicity will be equal to a certain plethysm multiplicity. Later, we will use this fact to re-interpret the main theorem of \cite{LhouWillenbring} in terms of the branching from $GL_n$ to $\mathfrak{S}_n$.

We regard $F^{\la}$ as a functorial operator on the category of vector spaces -- often called the \textit{Schur functor}. This point of view applies when vector spaces are infinite dimensional. For example, if $\la = (1,1,1,...,1)$ (with $j$ 1's), then $F^{\la}$ takes a vector space to its $j$-th exterior power.  This situation can be generalized to the situation where $\la$ is an arbitrary non-negative partition. In the finite dimensional case, the situation is clear: if $V$ is a vector space of dimension n, then $F^{\la}(V) \cong F^{\la}_n$.

Given a partition $\la$ with at most $n$ parts, and a partition $\mu$ of $n$, we may consider the multiplicity, denoted in short-hand by the coefficient $b^\la_\mu$, or in brackets as follows:
\begin{equation*}
    b^\la_\mu = \left[ Y^\mu_n : F^\la_n \right] = \dim \mbox{Hom}_{\mathfrak{S}_n}( Y^\mu_n , F^\la_n)
\end{equation*}
where the $F^\la_n$ is regarded as a $\mathfrak{S}_n$-representation by restricting to the permutation matrices.  These multiplicities are impossible to compute in any general way, but as we mentioned earlier there are algorithms.

We present here a way to describe $\left[ Y^\mu_n : F^\la_n \right]$ using Schur functors. That is, if $V$ is a complex vector space of dimension $n$ we will use the notation of the Schur functor, $F^\la(V)$ to denote the irreducible $GL(V)$ representation.  So, for example if $V$ is $n$-dimensional then by identifying $GL(V)$ with $GL_n$, we have $F^\la_n \equiv F^\la(V)$.  Then recall that $F^\la(V)$ can be defined for infinite dimensional $V$.

For a vector space $W$, let $Sym W$ denote the algebra of symmetric tensors on $W$,
which is a graded $GL(W)$ representation. (Recall, if $W$ is finite dimensional, $Sym W$ is isomorphic, as a ring, to the polynomial functions on $W^*$.) The following theorem is well-known, but we include a sketch of it here to aid in the exposition. In fact, this is an exercise in Stanley's book \cite[Exercise 7.74]{Stanley2} with a (different) solution \cite[Page 534]{Stanley2} sketched there as well. This equation also appears in \cite[Theorem 5.1]{Scharf-Thibon}.

\begin{proposition}\label{thm:stanleyexercise}
    Given positive integers $k$ and $n$, fix a partition $\mu$ of $n$.  Regard $Sym \C^k$ as a graded $GL_k$-representation.  Then, the (infinite dimensional) representation $F^\mu(Sym \C^k)$ decomposes into irreducible finite dimensional representations of $GL_k$ with finite multiplicities, and for any partition $\la$ with at most $k$ parts,
    \begin{equation*}
        \left[ Y^\mu_n : F^\la_n \right] = \left[ F^\la_k : F^\mu(Sym \C^k) \right].
    \end{equation*}
\end{proposition}

\begin{proof}[Sketch of proof]
Let $V=Sym \C^k$. The tensor product of $n$ copies of $V$ may be regarded as a $GL(V) \times \mathfrak{S}_n$-representation with multiplicity free decomposition,
\begin{equation}\label{eq:1}
    \underbrace{V \ot \cdots \ot V}_{\mbox{$n$ copies.}} = \bigoplus F^\mu(V) \ot Y^\mu_n
\end{equation}
where the sum is over all partitions, $\mu$, of size $n$ (by Schur-Weyl duality\footnote{The concern here is that $V$ is infinite dimensional.  The reason that this proof is only a sketch is because of this technical point, however, it is enough because all the representations considered in this paper are graded finite-dimensional. For a very careful exposition of the foundations of plethysm we recommend the paper by Loehr and Remmel \cite{LoehrRemmel}.} applied to $V$). On one hand, we can restrict from $GL(V)$ to $GL_k$, which involves decomposing the $F^\mu(V)$ into irreducible representations, $F^{\la}_k$, of $GL_k$. On the other hand, observe that,
\begin{equation*}
    V \ot \cdots \ot V \cong Sym( \bbC^k \otimes \bbC^n ).
\end{equation*}
The right hand side carries an action of $GL_k \times GL_n$,
which by Howe duality decomposes as
\begin{equation*}
    Sym(\bbC^k \ot \bbC^n) \cong \bigoplus F^{\la}_k \ot F^{\la}_n
\end{equation*}
where the sum ranges over all partitions $\la$ with at most $\min(k,n)$ parts. We then branch from the right-hand $GL_n$ to $\mathfrak{S}_n$ to obtain
\begin{equation*}
    Sym(\bbC^k \ot \bbC^n) = \bigoplus_\la  F^{\la}_k \ot
    \left( \bigoplus_\mu \left[ Y^\mu_n : F^\la_n \right] Y^{\mu}_n \right).
\end{equation*}
Reorganizing we have
\begin{equation}\label{eq:2}
    Sym(\bbC^k \ot \bbC^n) = \bigoplus_\mu
        \left(\bigoplus_\la \left[ Y^\mu_n : F^\la_n \right] F^{\la}_k \right) \ot Y^{\mu}_n.
\end{equation}
Compare the decompositions \ref{eq:1} and \ref{eq:2}.
\end{proof}

\section{Branching for symmetric powers}\label{sec:OneRow}
In the previous section we saw that finding certain plethysm multiplicities was equivalent to finding certain branching multiplicities.
\begin{equation*}
     \left[F^\la_k : F^\mu(Sym \C^k)\right] = \left[Y^\mu_n : F^\la_n\right]
\end{equation*}
In this section we will consider what is already known in the literature (for example in \cite{Stanley2}) addressing the case when $k=l(\la)=1$. In this case, we will see that the branching multiplicities are already known for any irrep of $GL_n$ given by $F^\la_n$ where $\la=(d,0,0,\dots) = (d)$. These irreps correspond to symmetric powers of the defining representation of $GL_n$, denoted $Sym^d \C^n$. 

As an example, consider decomposing irreps of $GL_4$ into a direct sum of irreps of $\mathfrak{S}_4$. But as described above, only decompose the irreps of $GL_4$ given by $F^{(d)}_4$, which are equivalent to the $d$th symmetric power of $\C^4$. We have the symmetric group on four letters $\mathfrak{S}_4$ and its irreducible representations $Y^\mu_4$ where $\mu$ is from the set
\begin{equation*}
    \mu \in \left\{ (4), (3,1), (2,2), (2,1,1), (1,1,1,1) \right\}
\end{equation*}
Consider the first few. How do they decompose into irreps of $\mathfrak{S}_4$?
\begin{align*}
    Sym^0 \C^4 = F^{(0,0,0,0)}_4 &= Y_4^{\tiny\yng(4)}\\
    Sym^1 \C^4 = F^{(1,0,0,0)}_4 &= Y_4^{\tiny\yng(4)} \oplus Y_4^{\tiny\yng(3,1)}\\
    Sym^2 \C^4 = F^{(2,0,0,0)}_4 &= 2Y_4^{\tiny\yng(4)} \oplus 2Y_4^{\tiny\yng(3,1)} \oplus Y_4^{\tiny\yng(2,2)}\\
    Sym^3 \C^4 = F^{(3,0,0,0)}_4 &= 3Y_4^{\tiny\yng(4)} \oplus 4Y_4^{\tiny\yng(3,1)} \oplus Y_4^{\tiny\yng(2,2)} \oplus Y_4^{\tiny\yng(2,1,1)}
\end{align*}
These results can be obtained using simple combinatorial rules. The multiplicity of a given $Y^\mu_4$ inside $Sym^d \C^4$ is given by the number of semi-standard tableaux (weakly increasing along the row and strictly increasing down the column) with total weight summing to $d$. For example, with $Sym^2 \C^4$ we count the semi-standard tableaux of with total weight $2$ for each shape $\mu$. These are
\begin{equation*}
    \young(0002), \young(0011), \young(001,1), \young(000,2), \young(00,11).
\end{equation*}
This gives us the decomposition of $F^{(2)}_4$ into irreps of $\mathfrak{S}_4$. Of course, these rules only apply to the one-row irreps $F^\la_n$ where $l(\la) \leq 1$. There are of course many more irreps of $GL_n$ whose branching decompositions are unknown in general. In this paper, we obtain results pertaining to the next simplest case, the irreps of $GL_n$ given as $F^\la_n$ for $l(\la) \leq 2$, which is just another way of saying we will look at two-row shapes.

\textbf{Remark:} These results explain the location of the last X on the $\lambda_1$ axis in Figure \ref{fig:curiousCurve}.  The multiplicity of $Y_n^{\mu}$ in $F_n^{(d)}$ is nonzero for all $\mu$ if and only if $d \geq \binom{n}{2}$, so the X's stop there.

\color{black}
\textbf{Remark:} In \cite[page 475]{Stanley2} the formula expressing $\left[Y^\mu_n : F^{(d)}_n\right]$ can be found by examining the coefficients of the identity
\begin{equation*}
    \sum_{d \geq 0} (\text{ch } \Psi^d) q^d = \sum_{\mu \vdash n} s_\mu (1, q, q^2, \dots ) s_\mu
\end{equation*}
where
\begin{equation*}
    s_\mu (1, q, q^2, \dots ) = \frac{q^{b(\mu)}}{\prod_{u \in \mu} [h(u)]}.
\end{equation*}
More details are contained in \cite{Stanley2} but briefly, $\Psi^d$ is the character of the action of the symmetric group induced on degree $d$ homogeneous forms on an $n$-dimensional vector space. As for the other pieces of notation, $b(\mu) = \sum (i-1) \mu_i$ and $u \in \mu$ means that we are identifying $\mu$ with its diagram $\{ (i,j) \, : \, 1 \leq j \leq \mu_i \}$ so that $u$ is a square whose hook length $h(u) = \mu_i + \mu_j^{'} - i - j + 1$. Finally $[k] = 1 - q^k$.
\color{black}

\section{Using an existence result on the plethysm side}\label{sec:finishProof}

In the previous section we saw the results for one-row diagrams, what we also call the $k=\ell(\la)=1$ case, already known in the literature. This dealt with decomposing certain irreps $F^\la_n$ of $GL_n$ into irreps $Y^\mu_n$ of $\mathfrak{S}_n$ in the case that $\la=(d,0,0,\dots)$. In Section \ref{sec:ConnectBranchingPlethysm} we saw that finding the multiplicity of $F^\la_k$ inside $F^\mu(Sym \C^k)$ (plethysm) was equivalent to finding the multiplicity of $Y^\mu_n$ inside $F^\la_n$ (branching). In this section we will apply the main theorem from \cite{LhouWillenbring} on the plethysm side to guarantee non-zero multiplicity of certain irreps when $k=l(\la)=2$. Thus we will also have guaranteed non-zero multiplicity of certain branching multiplicities as well.

\begin{lemma}
    Let $Sym^m \C^k$ be the degree $m$ symmetric tensors on $\C^k$. Applying the Schur functor, we regard $F^\mu(Sym^m \C^k)$ as a $GL_k$-representation. If the multiplicity of $F^\la_k$ in $F^\mu(Sym^m \C^k)$ is non-zero, then its multiplicity in $F^\mu(Sym \C^k)$ is non-zero as well.
\end{lemma}

\begin{proof}
Since we have the injection $Sym^m \C^k \xhookrightarrow{} Sym \C^k$, we also have the injection $F^\mu(Sym^m \C^k) \xhookrightarrow{} F^\mu(Sym \C^k)$. This follows by considering generalized Littlewood-Richardson coefficients. The notation gets a bit trickier here, but briefly, when we decompose an irrep of $GL(V_1 \oplus V_2)$ into irreducibles under the action of a subgroup $GL(V_1) \times GL(V_2)$, the multiplicities that show up are Littlewood-Richardson coefficients
\begin{equation*}
    F^\mu(V_1 \oplus V_2) = \bigoplus c^\mu_{\lambda \nu} F^\lambda(V_1) \otimes F^\nu(V_2).
\end{equation*}
Replacing $V_1 \oplus V_2$ by the direct sum $\bigoplus_{1,\dots,r} Sym^i(V)$ in the left-hand side we obtain an updated right-hand side
\begin{equation*}
    \bigoplus c^\mu_{\Vec{\mu}} \big( F^{\mu_1}(Sym^0(V)) \otimes \cdots \otimes F^{\mu_r}(Sym^{r-1}(V)) \big)
\end{equation*}
where the sum is over all tuples of partitions $\Vec{\mu} = (\mu_1,\mu_2,\dots ,\mu_r)$ where, somewhat confusingly, each $\mu_i$ is now a partition (rather than a natural number). Taking the trivial representation in all tensor factors except the $m$th, where we take $\mu_m=\mu$, we get one of the terms in this direct sum. Since we are assuming $F_k^{\la}$ appears with non-zero multiplicity in $F^\mu(Sym^m \C^k)$, and since Littlewood-Richardson coefficients also describe tensor product multiplicities, that particular coefficient is non-zero, and our result follows.
\end{proof}

From here onwards, set $k=2$. This allows us to look at the finite-dimensional representation $F^\mu(Sym^m \C^2)$ for any $m \in \mathbb{N}$. Thus for every choice of $m \in \mathbb{N}$ we can hope for results to translate back to branching multiplicities.

\begin{lemma}\label{lem:center}
    Fix any partition $\mu$ of size $n$ and any $m \in \mathbb{N}$. Given the $GL_2$ representation $F^\mu(Sym^m \C^2)$, the center of $GL_2$ acts by
    \begin{equation*}
        v \longmapsto z^{nm}v
    \end{equation*}
    where $v \in F^\mu(Sym^m \C^2)$, $z \in \C^\times$, and $diag(z,z) \in Z(GL_2)$.
\end{lemma}

\begin{proof}
    This follows by homogeneity of the Schur function, $s_\la(t\Vec{x}) = t^{\abs{\la}}s_\la(\Vec{x})$.
\end{proof}

\begin{lemma}
    Fix $\abs{\mu}=n$ and $m \in \mathbb{N}$. If $F^\la_2$ has non-zero multiplicity inside $F^\mu(Sym^m \C^2)$ then
    \begin{equation*}
        \la = (d+p,p) \mbox{ where } 2p+d=nm \mbox{ for some } p,d \in \mathbb{N}.
    \end{equation*}
\end{lemma}

\begin{proof}
    Recall the irreps of $GL_2$ are given by $\mbox{det}^p \ot Sym^d \C^2$ where the center $Z(GL_2)$ acts by $z^{2p+d}$. Then by Lemma \ref{lem:center} we must have $2p+d = nm$.
\end{proof}

\begin{proposition}\label{thm:translateJebHassan}
    Choose any partition $\abs{\mu}=n$ and any $m \in \{2,3,\dots\}$. Consider all (finitely many) irreducible representations $F^\la_2$ where $\la = (d+p,p)$ for some $d \in \{0,1,2,\dots,m \}$, where $2p+d=nm$ for some $p \in \mathbb{N}$. For at least one such $\la$
    \begin{equation*}
        \left[ F^\la_2 : F^\mu(Sym^m \C^2) \right] \ne 0.
    \end{equation*}
\end{proposition}

\begin{proof}
    We have $F^\mu(Sym^m \C^2)$ a representation of $GL_2$ via composition, but it is also an irrep of $GL_{m+1}$ since the Schur functor is being applied to an $m+1$-dimensional vector space. Restricting to $SL_{m+1}$ it is again irreducible. By the main theorem of \cite{LhouWillenbring} we are guaranteed the existence of some subrepresentation isomorphic to $Sym^d \C^2$, an irrep of $SL_2$, for some $d \in \{0,1,2,\dots,m \}$. This works provided $m\geq2$ because they find principal embeddings of $\mathfrak{sl}_2$ inside $\mathfrak{sl}_a$ for $a\geq 3$. This extends to some irrep of $GL_2$, which is given by $\mbox{det}^p \ot Sym^d \C^2$ for some $p \in \mathbb{N}$. But by Lemma \ref{lem:center} we must also have $2p+d=nm$.
\end{proof}

\begin{flushleft}
\textbf{Remark:} The parity of $nm$ and $d$ must match. Solving for $p$ we see that $p = \frac{nm-d}{2}$, but $p \in \mathbb{N}$.
\end{flushleft}

\begin{proposition}\label{thm:Main}
    Choose any irreducible representation $Y^\mu_n$ of $\mathfrak{S}_n$ and choose any $m \in \{2,3,4,\dots\}$. Consider the set of irreducible representations of $GL_n$ denoted $F^\la_n$ where $\la = (p+d,p)$, $d \in \{0,1,2,\dots,m\}, p \in \mathbb{N}$ and $2p+d=nm$. Then the multiplicity
    \begin{equation*}
        \left[ Y^\mu_n : F^\la_n \right] \ne 0
    \end{equation*}
    for at least one of the $F^\la_n$.
\end{proposition}

\begin{proof}
    This follows from Theorem \ref{thm:stanleyexercise} and Theorem \ref{thm:translateJebHassan}.
\end{proof}

\bibliographystyle{unsrt}
\bibliography{sample}

\begin{thebibliography}{10}

\bibitem{Kostant}
Bertram Kostant.
\newblock The principal three-dimensional subgroup and the {B}etti numbers of a
  complex simple {L}ie group.
\newblock {\em Amer. J. Math.}, 81:973--1032, 1959.

\bibitem{LhouWillenbring}
Hassan Lhou and Jeb~F. Willenbring.
\newblock Lowest sl(2)-types in sl(n)-representations.
\newblock {\em Represent. Theory}, 21:20--34, 2017.

\bibitem{WillenbringZuckerman}
Jeb~F. Willenbring and Gregg~J. Zuckerman.
\newblock Small semisimple subalgebras of semisimple {L}ie algebras.
\newblock In {\em Harmonic analysis, group representations, automorphic forms
  and invariant theory}, volume~12 of {\em Lect. Notes Ser. Inst. Math. Sci.
  Natl. Univ. Singap.}, pages 403--429. World Sci. Publ., Hackensack, NJ, 2007.

\bibitem{King1974}
R.~C. King.
\newblock Branching rules for {${\rm GL}(N)\supset \mathfrak{S}_{m}$} and the
  evaluation of inner plethysms.
\newblock {\em J. Mathematical Phys.}, 15:258--267, 1974.

\bibitem{Nishiyama}
Kyo Nishiyama.
\newblock Restriction of the irreducible representations of $gl_n$ to the
  symmetric group $\mathfrak{S}_n$.
\newblock {\em http://rtweb.math.kyoto-u.ac.jp/home\_kyo/preprint/glntosn.pdf}.

\bibitem{OrellanaZabrocki2019}
Rosa Orellana and Mike Zabrocki.
\newblock Products of symmetric group characters.
\newblock {\em J. Combin. Theory Ser. A}, 165:299--324, 2019.

\bibitem{Etingof}
Pavel Etingof, Oleg Golberg, Sebastian Hensel, Tiankai Liu, Alex Schwendner,
  Dmitry Vaintrob, and Elena Yudovina.
\newblock {\em Introduction to representation theory}, volume~59 of {\em
  Student Mathematical Library}.
\newblock American Mathematical Society, Providence, RI, 2011.
\newblock With historical interludes by Slava Gerovitch.

\bibitem{Weyl1}
Hermann Weyl.
\newblock {\em The classical groups}.
\newblock Princeton Landmarks in Mathematics. Princeton University Press,
  Princeton, NJ, 1997.
\newblock Their invariants and representations, Fifteenth printing, Princeton
  Paperbacks.

\bibitem{Stanley2}
Richard~P. Stanley.
\newblock {\em Enumerative combinatorics. {V}ol. 2}, volume~62 of {\em
  Cambridge Studies in Advanced Mathematics}.
\newblock Cambridge University Press, Cambridge, 1999.
\newblock With a foreword by Gian-Carlo Rota and appendix 1 by Sergey Fomin.

\bibitem{Scharf-Thibon}
T.~Scharf and J.Y. Thibon.
\newblock A hopf-algebra approach to inner plethysm.
\newblock {\em Advances in Mathematics}, 104(1):30 -- 58, 1994.

\bibitem{LoehrRemmel}
Nicholas~A. Loehr and Jeffrey~B. Remmel.
\newblock A computational and combinatorial expos\'{e} of plethystic calculus.
\newblock {\em J. Algebraic Combin.}, 33(2):163--198, 2011.

\end{thebibliography}

Alexander Heaton, Max Planck Institute for Mathematics in the Sciences, Leipzig and Technische Universit\"at Berlin, Germany \\{\tt E-mail address:} alexheaton2@gmail.com, heaton@mis.mpg.de
\vspace{0.3 cm} \\
Songpon Sriwongsa, Department of Mathematics, Faculty of Science, King Mongkut's University of Technology Thonburi (KMUTT), Bangkok 10140, Thailand \\{\tt E-mail address:} songpon.sri@kmutt.ac.th
\vspace{0.3 cm} \\
Jeb F. Willenbring, Department of Mathematical Sciences, University of Wisconsin-Milwaukee, United States \\{\tt E-mail address:} jw@uwm.edu

\end{document}